\documentclass{amsart}
\usepackage{va}

\usepackage[pdftex]{hyperref}
\hypersetup{colorlinks,citecolor=blue,linktocpage,hyperindex=true,backref=true}

\usepackage{bm, tikz-cd}

\usepackage{todonotes}

\newcommand\ocM{\overline{\cM}}
\newcommand\SP{\operatorname{SP}}
\newcommand\cSP{\operatorname{\cS\cP}}

\newcommand\us{^{\rm s}}
\newcommand\uss{^{\rm ss}}
\newcommand\uperm{^{\rm p}}
\newcommand\eps{\epsilon}
\newcommand\Sym{\operatorname{Sym}}
\newcommand\sign{\operatorname{sign}}

\newcommand\bry{_{\rm bry}}
\newcommand\inr{_{\rm int}}

\newcommand\hor{^{\rm hor}}
\newcommand\ver{^{\rm ver}}
\newcommand\cam{_{\rm Cam}}

\newcommand\ubur{^{\rm Bur}}

\newcommand\tor{^{\rm tor}}

\title{Kappa classes on KSBA spaces}
\date{September 28, 2023} 
\author{Valery Alexeev}
\email{valery@uga.edu}
\address{Department of Mathematics, University of Georgia, Athens GA 30602, USA}

\begin{document}
\begin{abstract}
  We define kappa classes on moduli spaces of KSBA stable
  varieties and pairs, generalizing the Miller-Morita-Mumford classes
  on moduli of curves, and compute them in some cases where the
  virtual fundamental class is known to exist, including Burniat and
  Campedelli surfaces.  For Campedelli surfaces, an intermediate step
  is finding the Chow (same as cohomology) ring of the GIT quotient
  $(\bP^2)^7//SL(3)$.
\end{abstract}

\maketitle

\section{Introduction}
\label{sec:intro}

The Miller-Morita-Mumford (MMM) classes, or kappa classes, are some
very basic objects in enumerative geometry of the moduli spaces
$\oM_{g,n}$ of stable curves. For example, according to Mumford's
conjecture proved by Madsen and Weiss \cite{madsen2007stable-moduli},
the stable cohomology group of $M_g$ is
$\bQ[\kappa_1,\kappa_2,\dotsc]$.  These classes were introduced by
Mumford in \cite{mumford1983towards-enumerative}.  Morita
\cite{morita1987characteristic-classes} defined equivalent classes on
$M_g$ from a topological point of view, and Miller
\cite{miller1986homology-mapping} showed that
$\bQ[\kappa_1,\kappa_2,\dotsc]$ embeds into the stable cohomology of
$M_g$ in degrees $\le g/3$.

In \cite{donaldson2020fredholm-topology} Donaldson asked if it were
possible to extend enumerative geometry of $\oM_g$ to the KSBA
spaces, the moduli spaces of stable varieties which are
higher-dimensional analogues of stable curves. He outlined a
definition of the virtual fundamental class on the moduli space of
stable surfaces, which was subsequently developed by Jiang
\cite{jiang2022virtual-fundamental}.

Below, we extend the definition of the kappa classes to the KSBA
spaces and ask some basic questions about them. The rest of the paper
is devoted to computing them in several cases where the moduli spaces
of stable surfaces are known explicitly, such as products of curves
\cite{vanopstall2006stable-degenerations1}, Campedelli surfaces
\cite{alexeev2009explicit-compactifications} and Burniat surfaces
\cite{alexeev2009explicit-compactifications,
  alexeev2023explicit-compactifications}.  We work over $\bC$.

\begin{acknowledgements}
  I thank Michel Brion, Yunfeng Jiang and Yuji Odaka for helpful
  conversations, and especially Bill Graham for teaching me the basics
  of equivariant cohomology. The author was partially supported by NSF
  under DMS-2201222.
\end{acknowledgements}

\section{Definition of kappa classes}
\label{sec:kappa}

A KSBA-stable pair $(X, D=\sum_i a_iD_i)$ consists of an
equidimensional variety~$X$ and integral Weil divisors $D_i$ taken
with rational coefficients $0<a_i\le 1$, such that $X$ is deminormal
(and in particular has only double crossings in codimension~$1$), $D_i$
are Mumford divisors (so do not contain components of the double locus
of $X$), the pair $(X,D)$ has semi log canonical singularities, and
the divisor $K_X+D$ is ample.

We refer to Koll\'ar's
\cite[Def.\ 8.13]{kollar2023families-of-varieties} for the definition of
the moduli functor which is quite delicate and involves an important
notion of a K-flat family of Mumford divisors.  The main result of
\cite{kollar2023families-of-varieties} is that after fixing the basic
invariants: dimension $d=\dim X$, the coefficient set $\bm{a} = (a_i)$
and the volume $\nu = (K_{X}+D)^d$, this moduli functor admits
a projective coarse moduli space $\SP(\bm{a}, d, \nu)$. The moduli
stack $\cSP(\bm{a}, d, \nu)$ is a proper Deligne-Mumford stack. 

\begin{lemma}
  For any family $f\colon (X,D)\to S$ of KSBA-stable pairs, there
  exists a well defined $\bQ$-line bundle $K_{X/S}+D$ on $X$ which is
  functorial, i.e. compatible with base change $S'\to S$.
\end{lemma}
\begin{proof}
It is known that
there exists an open subset $j\colon U\to X$ such that
\begin{enumerate}
\item For any fiber $X_s$, one has $\codim X_s\setminus U_s\ge2$.
\item $U\to S$ is Gorenstein and fibers have at worst simple double
  crossings.
\item The divisors $D_i|_U$ are Cartier and lie in the smooth locus of $U$.
\item For some $N\in\bN$, $Na_i\in\bZ$ and the sheaf
  $L_N:=j_*\big(\omega_{U/S}^{\otimes N}(ND)\big)$ is invertible. 
\end{enumerate}
Define $K_{X/S}+D:=\frac1{N}L_N$.  This definition is
independent of taking further multiples of $N$ and choosing another
open subset $U$ with the above properties.

For any base change $S'\to S$, the opens set 
$j'\colon U'=U\times_S S' \to X'=X\times_S S'$ has the same properties
and $L_N'=j'_*\big(\omega_{U'/S'}^{\otimes N}(ND')\big)$ is the
pullback of $L_N$, because formation of
$\omega_{U/S}$ and $\cO_U(D_i)$ commutes with base changes.
\end{proof}

\begin{corollary}
  On the universal family $(\cX,\cD)\to \cSP(\bm{a}, d, \nu)$ over the
  moduli stack there is a canonical $\bQ$-line bundle $K_{\cX/\cSP}+\cD$.
\end{corollary}

Mumford \cite{mumford1983towards-enumerative} defined the kappa
classes $\kappa_i$ on $\ocM_g$ as the pushforwards of the cycles
$K_{\cX/\ocM_g}^{i+1}$ in the universal family
$f\colon \cX\to \ocM_g$.  Similarly, Arbarello and Cornalba
\cite{arbarello1996combinatorial, arbarello1998calculating-cohomology}
defined $\kappa_i$ on $\ocM_{g,n}$ as the pushforwards of 
$(K_{\cX/\ocM_{g,n}}+~\cD)^{i+1}$ in the universal family
$f\colon (\cX,\cD=\sum_{k=1}^n\cD_k)\to \ocM_{g,n}$. In both case we
are greatly helped by the fact that $\ocM_g$ and $\ocM_{g,n}$ are
smooth Deligne-Mumford stacks, so $\kappa_i$ can be considered to be
cocycles in $A^i(\ocM_{g,n})_\bQ$ and $H^{2i}(\ocM_{g,n},\bQ)$.

We would like to define the kappa classes on KSBA spaces similarly, as
the pushforwards of $(K_{\cX/\cSP}+\cD)^{i+d}$. The question is: in
what generality does this definition make sense and which properties
does it have? We propose several versions.

\begin{definition}[$\kappa_0$ and $\kappa_1$]
  Obviously, for any family $\kappa_0$ can be defined simply as
  $\nu=(K_{X_s}+D_s)^d\in\bQ$, the volume of a general fiber.
  $\kappa_1 = f_*(K_{X/S}+D)^{d+1}$ is known as the CM line
  bundle, see e.g. \cite{fine2006note-on-positivity,
    wang2014nonexistence-asymptotic, patakfalvi2017ampleness-cm}
\end{definition}

\begin{definition}[Cycle version]
  For any family $f\colon (X,D)\to S$ with equidimensional base
  $S$, define the cycles $\kappa_i(S)\in A_{\dim S-i}(S)$ as proper
  pushforwards 
  \begin{displaymath}
    \kappa_i(S) = f_*\left(
      (K_{X/S} + D)^{i+d} \cap [X] \right)
    \quad\text{under}\quad
    \quad f_*\colon A_{m}(X) \to A_{m}(S),
  \end{displaymath}
  where $m = \dim X-i-d = \dim S-i$. Consider a commutative square
  \begin{center}
    \begin{tikzcd}
      (X,D) \arrow{d}[swap]{f} & (X',D')\arrow{d}{f'} \arrow{l}[swap]{g'} \\
      S & S' \arrow{l}{g}
    \end{tikzcd}    
  \end{center}
  $g\colon S'\to S$ is a proper generically finite morphism of
  degree $e$, with reduced $S,S'$.
  Then $g'$ is also a proper generically finite morphism
  of degree $e$. We have that $g'_*[X'] = e[X]$, and by the projection formula
  \begin{eqnarray*}\label{eq:com-square}
    e \kappa_i(S) &=& 
                      f_*\left( (K_{X/S} + D)^{i+d} \cap g'_* [X']\right)\\
                  &=&
                      f_*g'_* \left( (K_{X'/S'} + D')^{i+d} \cap [X'] \right) \\
                  &=& g_* f'_* \left( (K_{X'/S'} + D')^{i+d} \cap [X'] \right)
                      = g_* \kappa_i(S')
  \end{eqnarray*}

  The same definition and with the same functoriality also works for
  the morphisms of DM stacks using the intersection theory on stacks
  \cite{vistoli1989intersection-theory}. In particular, let $\cSP'$ be
  an irreducible component of $\cSP_{\rm red}$, or its normalization.
  Then we get cycles
  $\kappa_i(\cSP')\in A_{\dim\cSP'-i}(\cSP')$ and 
  $\kappa_i(\SP')\in A_{\dim\SP'-i}(\SP')$ on its coarse moduli space.
\end{definition}

\begin{definition}[Smooth moduli stack version]
  \label{def:smooth-stack}
  If $\cSP'$ is a smooth (necessarily proper) Deligne-Mumford stack
  then, just as for $\ocM_{g,n}$ we can identify the group
  $A_{\dim(\cSP')-i}(\cSP')_\bQ$ with $A^i(\cSP')_\bQ$ and define
  $\kappa_i$ in $A^i(\cSP')_\bQ$. If $\cSP'$ is a global quotient of a
  smooth projective variety $[V:\Gamma]$ by a finite group then
  $A^i(\cSP')_\bQ = A^i(V)_\bQ^\Gamma$. Using the cycle map, we also
  get kappa classes in 
  $H^{2i}(\cSP',\bQ) = H^{2i}(V,\bQ)^\Gamma$.
\end{definition}

\begin{definition}[Lci morphisms]
  Suppose that the morphism $f\colon X\to S$ is lci (e.g., families of
  stable curves are lci) and that $S$ is smooth and proper.  Using the
  identification $A_{\dim S-i}(S) = A^i(S)$ we get the kappa classes
  in $A^i(S)$. If $g\colon S'\to S$ is an arbitrary morphism from
  another smooth and proper variety $S'$, the functoriality of the
  refined Gysin homomorphism in homology
  \cite{fulton1984intersection-theory} implies that for a base change
  $g\colon S'\to S$ one has $\kappa_i(S') = g^*\kappa_i(S)$. In
  particular, we get kappa classes on a resolution of singularities of
  $\cSP'$ in a functorial way.
\end{definition}

\begin{definition}[Cohomological version]
  Without assuming that $f$ is lci, for any family of stable pairs
  $f\colon (X,D)\to S$ over a smooth, not necessarily proper $S$, we
  can use Gysin pushforward $H^i(X,\bQ)\to H^i(S,\bQ)$ defined as the
  composition
  \begin{displaymath}
    H^{2(i+d)}(X) \xrightarrow{\cap [X]} H^{\rm BM}_{2\dim X-2i-2d}(X)
    \xrightarrow{f_*} H^{\rm BM}_{2\dim S-2i}
    \isoto H^{2i}(S),
  \end{displaymath}
  where $H^{\rm BM}_*$ is Borel-Moore homology, see
  e.g. \cite[App.~B]{fulton1997young-tableaux}. I don't know if this
  Gysin pushforward has enough functorial properties to imply 
  $\kappa_i(S') = g^*\kappa_i(S)$.   
\end{definition}

\begin{definition}[Almost lci morphisms]
  Few KSBA-stable surfaces have
  lci singularities. But for most of them the index-$1$ Gorenstein
  covers are lci. An important idea of Jiang
  \cite{jiang2022virtual-fundamental} is to utilize the DM stack of
  index-$1$ covers to define the virtual fundamental class of
  $\SP$. One may use the same idea here and define $\kappa_i$ as 
  Gysin pushforward of the pullback of $(K_{X/S}+D)^{i+d}$ from $X$ to
  its index-$1$ covering stack.  
\end{definition}

At this point an educated reader will certainly think of many other
ways to define kappa classes in cohomology, e.g. using operational
Chow rings and bivariant theories, \'etale Borel-Moore
homology, etc. I think all of them deserve a serious consideration.

In what follows I assume that we have well-defined kappa classes
in cohomology, in one of the above ways. In the three examples considered
later in this paper the moduli stacks are smooth and we use
Definition~\ref{def:smooth-stack} to compute $\kappa_i$.

\medskip

The semipositivity results for families of stable pairs
\cite{fujino2018semipositivity-theorems,
  kovacs2017projectivity-moduli} imply that $K_{X/S} + D$ is nef, by a
standard argument. It follows that whenever $\kappa_i$ are defined in
cohomology, they are nef.  In fact, $\kappa_1$, i.e. the CM line
bundle, is known to be ample.  This is true for $\oM_g$ by
\cite{mumford1977stability-of-projective}, for $\oM_{g,n}$ by
\cite{cornalba1993on-the-projectivity} and in general by
\cite{patakfalvi2017ampleness-cm}.

\medskip 

A special and very interesting case is the KSBA compactification of
the moduli of log Calabi-Yau pairs $(X,\Delta+\epsilon B)$ such that
generically $K_X+\Delta \sim_\bQ 0$ and $B$ is $\bQ$-Cartier and
ample. By \cite{kollar2019moduli-of-polarized,
  birkar2023geometry-polarized} after fixing basic numerical
invariants there exists $\epsilon_0>0$ such that for any
$0<\epsilon <\epsilon_0$ the compactification for the stable pairs
$(X,\Delta+\epsilon B)$ does not depend on $\epsilon$. Some concrete
cases are the compactified moduli spaces of toric and abelian
varieties \cite{alexeev2002complete-moduli} and of K3 surfaces
\cite{alexeev2023compact}.

\begin{definition}
  For a family $f\colon (X,\Delta+\epsilon B)\to S$ of KSBA-stable log
  Calabi-Yau pairs with $0<\epsilon\ll 1$, the generalized Hodge
  $\bQ$-line bundle $\lambda$ is defined by the condition
  $K_{X/S} + \Delta = f^*(\lambda)$. By functoriality this defines 
  $\bQ$-line bundles on the moduli stack and on its coarse moduli space.
\end{definition}

Obviously, $\kappa_i(\epsilon)$ are polynomials in $\epsilon$,
$\lambda$ and $f_* (B^{i+d})$. For example,
\begin{displaymath}
  \kappa_1(\epsilon) / \epsilon^d = 
  (n+1)\nu(B) \lambda+ \epsilon f_*B^{d+1},
\end{displaymath}
where $\nu(B) = B_s^d$ is the volume of a general fiber.  Nefness of
$\kappa_1(\epsilon)$ for $0<\epsilon\ll1$ implies that $\lambda$ is
nef as well.

\begin{conjecture}
  For KSBA-stable log Calabi-Yau pairs, $\lambda$ is semiample.
\end{conjecture}

For toric pairs $(X,\Delta + \epsilon B)$ with toric boundary
$\Delta$, one has $\lambda=0$. For abelian and K3 pairs
$(X,\epsilon B)$, $\lambda$ is the pullback of the ample Hodge bundle
on the Satake-Baily-Borel compactification. So it is true in these
cases. This was also proved for degenerations of pairs $(\bP^2,D)$ in
\cite{ascher2023moduli-boundary}.

\section{Products of curves}
\label{sec:products}

Consider surfaces of the form $X=C_g\times C_h$ which are products of
smooth curves of genus $g$ and $h$. Obviously, the stable limits of
one-parameter degenerations of such surfaces are products of two
stable curves. By van Opstall \cite{opstall2005moduli-of-products},
there is an irreducible component of the moduli of
stable surfaces isomorphic to $\oM_g \times \oM_h$ if
$g\ne h$, or a quotient of it by an involution if $g=h$.  
\cite{vanopstall2006stable-degenerations1} further extends this
construction to finite quotients of $C_g\times C_h$.

For a universal family $\cX$ over the stack $\ocM_g\times\ocM_h$ the
line bundle $\omega_{\cX/\cSP}$ is simply
$p_1^*(\omega_{\cC_g/\ocM_g}) + p_2^*(\omega_{\cC_h/\ocM_h})$ and the
kappa classes on $\oM_g\times\oM_h$ are merely appropriate sums of
monomials in the pullbacks of kappa classes from $\oM_g$ and
$\oM_h$. So this case is reduced to $\oM_g$.  The cases of quotients
of $C_g\times C_h$ can be treated similarly, on the stack level there
is not much difference.

\section{Kappa classes on moduli of $\bZ_2^k$-covers}
\label{sec:covers}

In the next two sections we treat the cases of Campedelli and Burniat
surfaces described in \cite{alexeev2009explicit-compactifications} and
\cite{alexeev2023explicit-compactifications}. These are surfaces of
general type that are certain branched $\bZ_2^k$-covers $(k=2,3)$ of
pairs $(Y, \frac12 D)$, $D=\sum_i D_i$. The stable surfaces on the
boundary are $\bZ_2^k$-covers of stable pairs $(Y,\frac12 D)$. In each case,
the compactified coarse moduli space of surfaces $X$ is a finite
quotient of the compactified fine moduli space for the pairs $(Y,\frac12 D)$
by a symmetry group permuting the labels of $D_i$'s.

Any family $f\colon X\to S$ after a finite base change $S'\to S$ can
be written as a $\bZ_2^k$-cover $X'\to (Y',\frac12 D')$, where
$X' = X\times_S S'$, $(Y',D') = (Y,D)\times_S S'$, and the $\bQ$-line
bundle $\omega_{X'/S'}$ is the pullback of
$\omega_{Y'/S'}(\frac12 D')$. Thus, the kappa classes for the covers
$X$ are proportional to the kappa classes for the pairs
$(Y,\frac12 D)$ by some multiples that are powers of $2$.  So it is
enough to study the kappa classes for the pairs $(Y,\frac12 D)$, which
we do in the next two sections.

\section{Burniat surfaces}
\label{sec:burniat}


Burniat surfaces are certain surfaces of general type of degree
$3\le d=K_X^2\le 6$ with $p_g=q=0$ which can be obtained as
$\bZ_2^2$-covers of degree-$d$ del Pezzo surfaces ramified in a set of
$12$ curves coming from a particular configuration of lines in
$\bP^2$.

Primary Burniat surfaces are those of degree~$6$,
they are covers of Cremona surface
$\Sigma = \Bl_3\bP^2$.  Secondary Burniat surfaces have degrees
$5$ and $4$, they are $\bZ_2^2$-covers of del Pezzo surfaces of degrees
$5$ and $4$ obtained by further blowups of $\Sigma$ at the points
where some three of the $12$ curves pass through the same point.

An explicit KSBA compactification of the moduli space of primary
Burniat surfaces was described in
\cite{alexeev2009explicit-compactifications}. Using it, explicit KSBA
compactifications for the moduli of secondary Burniat surfaces were
described in \cite{alexeev2023explicit-compactifications}.

For Burniat surfaces of degrees $6$ and $5$ and for the non-nodal
Burniat surfaces of degree~$4$, the above papers give compactifications
of the entire irreducible components in the moduli space of surfaces
of general type. In the nodal degree~$4$ and $3$ cases they are
closed subsets of irreducible components or larger dimensions. We 
do not discuss the nodal cases here.

As was pointed out to me by Yunfeng Jiang, for numerical
applications $4$ and $5$ are the most interesting degrees.
Indeed, for degree $d$ the dimension of the compactification $\oM\ubur_d$
is $d-2$. On the other hand, by \cite{donaldson2020fredholm-topology,
  jiang2022virtual-fundamental} the dimension of the virtual
fundamental class is $10-2d$. Thus, for $d=6$ the virtual fundamental
class is zero, and for $d=5$ it is a multiple of a point.

Below, we consider the degree $4$ non-nodal case. Here, the virtual
fundamental case has dimension~$2$ and coincides with
$[\oM_4\ubur]$. Below, we compute the kappa classes on $\oM_4\ubur$.

\subsection{Degree $6$}
\label{sec:burniat6}

Burniat surfaces $X_6$ of degree~$6$ are $\bZ_2^2$-covers of Cremona
surface $Y_6\tor = \Sigma=\Bl_3\bP^2$ ramified in a configuration of
$12$ curves shown in the left panel of Figure~\ref{fig-deg64}. A
$\bZ_2^2$-cover is determined by three divisors $R,G,B$ satisfying
certain conditions, see \cite{alexeev2009explicit-compactifications}.
We use the primary colors red, green and blue to draw them.  In this
case, $R=\sum_{i=0}^3 R_i$, $G=\sum_{i=0}^3 G_i$,
$B=\sum_{i=0}^3 B_i$. The curves with $i=0,3$ form the toric boundary
$D\bry$ of $\Sigma$. The curves with $i=1,2$ form the interior divisor
$D\inr$. The total branch divisor of $X_6\to Y_6$ is
$D_6=D\bry+D\inr = R+G+B$.

\begin{figure}[htbp]
  \centering
  \includegraphics{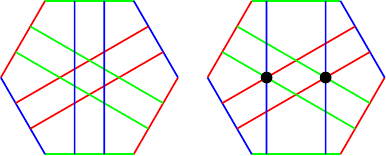} 
  \caption{Burniat configurations of degree 6 and 4 non-nodal cases}
  \label{fig-deg64}
\end{figure}

\cite{alexeev2009explicit-compactifications} constructs a compactified
moduli space for the pairs $(Y_6,\frac12 D_6)$, which we will denote
by $\oM_6$ here. It comes with a universal family
$(\cY_6\tor,\frac12 \cD_6) \to \oM_6$. Then the compactified moduli
space $\oM_6\ubur$ of degree $6$ Burniat surfaces is the quotient of
$\oM_6$ by a finite group $S_3\ltimes S_2^4$ shuffling the labels of
$R_i,G_i,B_i$.

In more detail, $(\cY_6,\frac12 \cD_6)\to \oM_6$ is obtained from an
explicit morphism of toric varieties $\cY_6\tor\to\oM_6\tor$ by a
series of smooth blowups, followed by a contraction to the relative
canonical model. The morphism $\oM_6\to \oM_6\tor$ is a composition of
a blowup $\rho_1$ at the central point
$1\in\bC^*{}^4\subset\oM_6\tor$, followed by a blowup $\rho_2$ along
six disjoint $\bP^1$. A family $\cY'_6\to\oM_6$ is obtained from
$\cY_6\tor$ by doing the base changes under $\rho_1,\rho_2$ and
additional smooth blowups in the fibers. On $\cY_6'$ the divisor
$K_{\cY_6'/\oM_6} + \frac12 D_6$ is relatively big and nef over
$\oM_6$. The universal family $\cY_6\to\oM_6$ is its relative
canonical model.

The boundary divisor $\cD\bry\tor$ is the union of the boundary curves
on the fibers, it is the horizontal part of the toric boundary of
$\cY_6\tor$. The interior divisor $\cD_6\tor$ on $\cY_6\tor$ is
constructed in \cite[Sec.\ 4]{alexeev2009explicit-compactifications}
as follows. In addition to the map $p_1\colon\cY_6\tor\to\oM_6\tor$,
there a second projection, a birational morphism
$p_2\colon\cY_6\tor\to V_{P_6}$ to a projective toric variety
$V_{P_6}$ defined by a lattice polytope $P_6$ that is the convex hull
of a $46$-point set $A_6$. Under this projection, the fibers $Y_6$
become closed subvarieties of $V_{P_6}$. Then $\cD\inr$ is a section
of $p_2^*\cO_{V_{P_6}}(2) \otimes p_1^*\cO_{\oM_6}(-F)$ for a certain
effective divisor $F$ on $\oM_6$ that is defined in the proof of
\cite[Prop.~4.20]{alexeev2009explicit-compactifications}.

\subsection{Degree~$4$}
\label{sec:burniat4}

Non-nodal Burniat surfaces of degree~$4$ are defined as follows. One
considers the special configurations for which the triples of the
curves $(R_1,G_1,B_1)$ and $(R_2,G_2,B_2)$ pass through common points,
as in the right panel of Figure~\ref{fig-deg64}.  Let $Y\to\Sigma$ be
the blowup at these points. The strict preimages of $R,G,B$ give a
$\bZ_2^2$-cover $\pi\colon X\to Y$ that is a Burniat surface of
degree~$4$. Note: the exceptional divisors are not included in the
branch divisor of $\pi$.

The compactified moduli space $\oM_4\ubur$ of degree~$4$ Burniat
surfaces was constructed in
\cite{alexeev2023explicit-compactifications} as an
$S_3\ltimes S_2^2$-quotient of the compactified moduli space $\oM_4$
of pairs $(Y,\frac12\cD)$, as follows. There exists a closed
subvariety $Z\subset\oM_6$, a complete intersection of two divisors,
over which the curves $(R_1,G_1,B_1)$ and $(R_2,G_2,B_2)$ are
incident.  The degenerate pairs appearing in this family are shown in
the upper row of Figure~\ref{fig-deg4}. The restricted family
$\cY_6|_Z$ comes with two disjoint sections $s_1,s_2$. Let $\cY'\to Z$
be the blowup of $\cY_6|_Z$ along $s_1,s_2$. 

The variety $Z$ is the strict preimage of a toric variety
$Z\tor \subset \cY\tor$ under the blowup $\rho_1$. It turns out that
$Z\tor\simeq\Sigma$ and $Z=\Bl_1\Sigma$. The divisor
$K_{\cY'/Z}+\frac12\cD'$ is relatively nef over $Z$; let $\cY''$ be its
relative canonical model. The degenerate fibers appearing in
$\cY''\to Z$ are shown in the lower row of Figure~\ref{fig-deg4}.
They are: a union of two $\bP^1\times\bP^1$ glued along the
diagonal, a union of four $\bP^2$, and another union of two
$\bP^1\times\bP^1$ with a different configuration of branch divisors.

\begin{figure}[htbp]
  \centering
  \includegraphics[width=300pt]{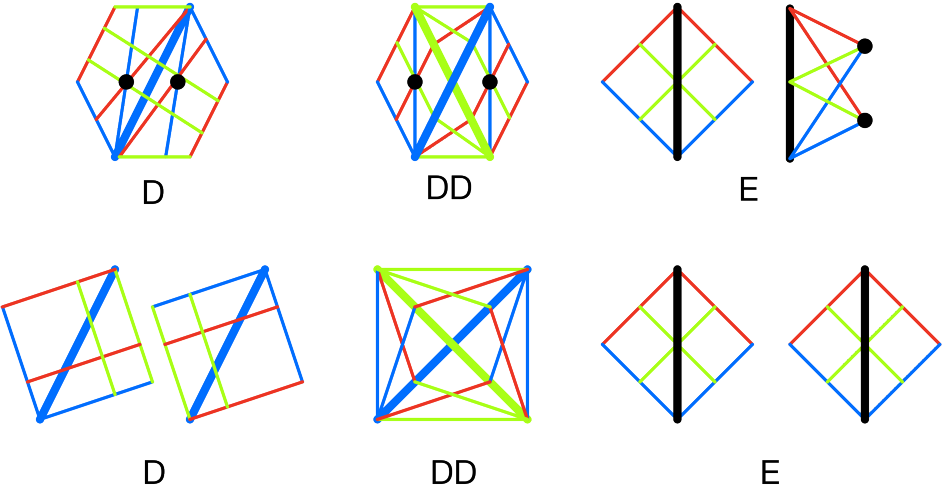} 
  \caption{Surfaces over $Z\subset\oM_6$ of degrees $6$ and $4$}
  \label{fig-deg4}
\end{figure}

Over the exceptional divisor of $\Bl_1\Sigma\to\Sigma$ (the divisor of
type E) all the fibers are isomorphic. Thus the family
$(\cY'',\frac12\cD'') \to Z$ descends to a family
$(\cY,\frac12\cD)\to \oM_4=Z\tor=\Sigma$.
Thus, the final family $\cY\to\Sigma$ is obtained from the toric family
$\cY\tor\to\Sigma$ as follows.  There is a sequence of smooth blowups
\begin{equation}\label{eq:blowups}
  \cY\tor = \cY_0
  \xleftarrow{\beta_1} \cY_1
  \xleftarrow{\beta_2} \cY_2
  \xleftarrow{\beta_3} \cY_3
  \xleftarrow{\beta_4} \cY_4 = \cY',
  \qquad\text{in which}
\end{equation}
\begin{enumerate}
\item $\beta_1$ is the blowup of the fiber $(\cY_6\tor)_1\simeq\Sigma$
  of $\cY_6\tor$ over $1\in\Sigma$. This is the base change $\cY_1 =
  \cY_0 \times_\Sigma \Bl_1\Sigma \to\cY_0$.
\item $\beta_2$ is the blowup of $\bP^1=\beta\inv(1)$, preimage of
  the central point $1\in(\cY_6\tor)_1$.
\item $\beta_3$ and $\beta_4$ are the blowups of the sections $s_i =
  R_i\cap G_i\cap B_i$, $i=1,2$.
\end{enumerate}
This sequence is followed by a contraction $\cY'\to\cY''$ followed by
a contraction $\cY''\to\cY$ covering the contraction $\Bl_1\Sigma\to\Sigma$.

\subsection{Kappa classes}
\label{sec:burniat-kappa}

\begin{theorem}
  In $A^*(\Sigma)$ one has $\kappa_0=1$,
  $\kappa_1 = \cO_\Sigma(1)=\cO(-K_\Sigma)$,
  $\kappa_2 = \tfrac{47}{4}\cdot[{\rm pt}]$.
\end{theorem}
\begin{proof}
  The class $\kappa_0$ is the degree of the divisor $K_Y+\frac12 D$
  on a general fiber, so
  \begin{math}
    \kappa_0 = (K_Y + \frac12 D)^2 = (-\frac12 K_Y)^2 = 1.
  \end{math}

  For the rest, we begin by computing the divisor
  $K_{\cY\tor/\Sigma} + \frac12 \cD\tor$. Let us denote the toric
  boundary of the toric variety $\cY\tor$ by $\Delta$. Denote the part
  of $\Delta$ that maps to the toric boundary of $\Sigma$ by
  $\Delta\ver$ and the remaining part by $\Delta\hor$. Obviously, one
  has $\Delta\ver = p_1^*(\Delta_\Sigma) = p_1^*\cO_\Sigma(1)$ for the
  projection $\cY\tor\to Z\tor =\Sigma$.

  As explained above, the family $\cY_6\tor\to\oM_6\tor$ comes with a
  second projection $p_2\colon\cY\tor_6\to V_{P_6}$ and the fibers
  $Y_6\tor$ are closed subvarieties sweeping out
  $V_{P_6}$. Restricting this family to $Z$ gives a family that sweeps
  out a smaller toric variety $V$ for the lattice polytope $P_4$
  obtained by an appropriate projection of $P_6$. By
  Lemma~\ref{lem:P4fam} the projection $\cY\tor_6\to V$ is small,
  since both varieties have $18$ toric boundary divisors. By
  Lemma~\ref{lem:P4}, $P_4$ is reflexive with a unique interior
  point. This implies that $\Delta = p_2^*\Delta_{V}$ and
  $-K_{V} = \Delta_{V} = \cO_{V}(1)$ . Restricting the divisor
  $\cD_{6,}{}\inr$ on $\oM_6\tor$ to the family over $Z\tor$ gives
  $\cO(\cD\inr) = p_2^* \cO_{V}(2) \otimes p_1^* \cO_\Sigma(-1)$.
  Putting this together, writing additively, and for convenience
  mixing up sheaves and divisors, we get:
  \begin{eqnarray*}
    &&K_{\cY\tor/\Sigma} = -\Delta\hor, \quad \Delta = p_2^*\cO_V(1),\quad
      \cD\bry = \Delta\hor,\\
    &&\cD\inr = p_2^*\cO_V(2) - p_1^*\cO_\Sigma(1) = 2\Delta - \Delta\ver,\\
    &&K_{\cY\tor/\Sigma} + \tfrac12 \cD\tor = 
       -\Delta\hor + \tfrac12(\Delta\hor + 2\Delta-\Delta\ver) =
       \tfrac12 \Delta  = p_2^*\cO_V(\tfrac12).
  \end{eqnarray*}
  By symmetry, $\kappa_1$ is a multiple of $\cO_\Sigma(1)$. To find
  this multiple it is
  enough to find its intersection with a boundary $(-1)$-curve $C$ on
  $\Sigma$, which equals $(K_\cY+\tfrac12 \cD)^3$ on the divisor
  $F = f\inv(C) \subset \cY$. To compute it, we can ignore the blowup
  $\rho_1$ since it does not touch $F$. We can also compute on the
  family $\cY'$ since the contraction $\cY'\to\cY''$ from a relative
  minimal model to a relative canonical model is crepant.

  The restriction of $\cY\tor$ to $F$ has two irreducible components
  corresponding to the two surfaces $\Bl_1\bF_1$ in the case D of
  Figure~\ref{fig-deg4}. Each component maps birationally to a
  boundary divisor of $V$.  Thus, the degree of $p_2^*\cO_V(1)$ on $F$
  is twice the degree of $\cO_V(1)$ on a boundary divisor of $V$. The
  latter degree is the lattice volume of the corresponding facet
  of $P_4$, which by Lemma~\ref{lem:P4} is $7$. So the degree of
  $K_{\cY\tor/\Sigma} + \frac12\cD$ on $F$ is $2\cdot \frac78$.

  Restricting $\cY\tor$ to $C$ gives a family $\cY\tor_C$ with two disjoint
  sections corresponding to the two special points. The family
  $\cY'_C$ is obtained from it by
  blowups at the two special sections, one in each irreducible
  component of $\cY\tor_C$. One has
  \begin{displaymath}
    K_{\cY'_C/C} + \tfrac12\cD|_F =
    \beta^*\big( K_{\cY_C \tor/ C} + \tfrac12\cD\tor_C \big) - \tfrac12 E_1
    -\tfrac12 E_2, 
  \end{displaymath}
  where $E_1$, $E_2$ are the exceptional divisors.  Using the blowup
  formula \cite[3.3.4]{fulton1984intersection-theory} and
  Lemma~\ref{lem:sections}, we get that the degree of
  $K_{\cY/\Sigma} + \frac12\cD$ on $F$ is
  $2\cdot \frac{7-3}{8} = 1$.  So,
  $\kappa_1 = \cO_\Sigma(1)$.

  To compute $\kappa_2 = (K_{\cY'/\Bl_1\Sigma} + \frac12\cD')^4$, we
  can compute on $\cY'\to\Bl_1\Sigma$ using the functorial
  property of the diagram~\eqref{eq:com-square}.  On $\cY\tor$ one has
  \begin{math}
    \left( K_{\cY\tor/\Sigma} + \frac12\cD\tor \right)^4 =
    \cO_V\left(\frac12\right)^4 =
    \frac{18\cdot 7}{2^4} = \frac{63}{8}
  \end{math}
  by Lemma~\ref{lem:P4}. Then we trace how this number changes under
  the four blowups in \eqref{eq:blowups} using 
  \cite[3.3.4]{fulton1984intersection-theory}.  
\end{proof}

The next two lemmas are proved by direct computations with polytopes.

\begin{lemma}\label{lem:P4}
  The polytope $P_4$ is a reflexive $4$-dimensional polytope with the
  $f$-vector $(1, 30, 84, 72, 18, 1)$ and a unique interior point. Its
  $18$ facets are isomorphic $3$-dimensional polytopes with $8$
  vertices and lattice volume $7$. One has $\vol(P_4) = 18\cdot 7$. 
\end{lemma}

\begin{lemma}\label{lem:P4fam}
  The toric family $\cY\tor$ is a projective toric variety for a
  $4$-dimensional lattice polytope $P_4+14 P_\Sigma$ with the
  $f$-vector $(1, 42, 96, 72, 18, 1)$.
\end{lemma}

Here, $P_\Sigma$ is the hexagon corresponding to the toric variety
$(\Sigma, \cO_\Sigma(1))$, and $14 P_\Sigma$ is the fiber polytope
coming from the construction of the toric family in
\cite{alexeev2009explicit-compactifications}.

\begin{lemma}\label{lem:sections}
  For each of the special sections $s_1=R\tor_1\cap G\tor_1\cap B\tor_1$
  and $s_2=R\tor_2\cap G\tor_2\cap B\tor_2$ of $\cY\tor\to\Sigma$, the
  normal bundle is trivial, i.e. equal $\cO_\Sigma^{\oplus2}$. 
\end{lemma}
\begin{proof}
  Consider the union $U$ of torus orbits in $\cY\tor$ containing
  $s_1$. Since $U$ comes with a section and a free action of the
  vertical torus $\bC^*{}^2 = \ker (\bC^*{}^4 \to \bC^*{}^2)$, one has
  $U = \Sigma\times \bC^*{}^2$. So, $N_{s_1/U} = \cO_{s_1}^{\oplus2}$,
  and the same works for $s_2$.
\end{proof}

\section{Campedelli surfaces}
\label{sec:campedelli}

Campedelli surfaces considered in
\cite{alexeev2009explicit-compactifications} are surfaces of general
type with $K^2=2$ and $p_g=0$ which are $\bZ_2^3$-covers of $\bP^2$
ramified in $7$ lines in general position. The branch data for this
cover consists of these seven lines $D_g$, $g\in\bZ_2^3\setminus 0$.
The moduli space has dimension~$6$, which coincides with the dimension
of the virtual fundamental class, equal $10\chi-2K_X^2$. So the
virtual fundamental class in this case is $[\ocM\cam]$.

By \cite{alexeev2009explicit-compactifications} the moduli stack
$\ocM\cam$ is a global quotient $[\oM:\Gamma]$, where $\oM$ is the
compactified moduli space of labeled log canonical pairs
$(\bP^2,\sum_{g\in\bF_2^3\setminus 0} \frac12 D_g)$ and
$\Gamma=\GL(3,\bF_2^3)$.  Further, $\oM$ is this case is the GIT
quotient $(\bP^2)^7//\SL(3)$ for the symmetric polarization
$(1,\dotsc,1)$.  The sets of stable and semistable points in this case
coincide and the $\SL(3)$-action on it is free. Therefore,
$(\bP^2)^7//\SL(3)$ is smooth.

The universal family $(\cY, \sum\frac12 D_g)\to\oM$ is itself a GIT quotient
of a family of hyperplane arrangements in
$(\bP^2)^7\times (\bP^2)^\vee$ by the action of the group $\SL(3)$ for
the polarization $(1,\dotsc, 1, \epsilon)$, $0<\epsilon\ll1$.

\smallskip

As a first step, I compute the Chow ring of
$\oM = (\bP^2)^7//\SL(3)$. Then the rational Chow ring of the stack
$\ocM\cam$ is identified with its $\Gamma$-invariant subring.

\subsection{Chow ring}
\label{sec:cohring}

Let $X=(\bP^2)^7$ with the diagonal action of $G=\SL(3)$ Consider
the GIT quotient $X//G$ for the symmetric polarization
$(1,\dotsc,1)$. In this case the stable and semistable loci coincide
and the $G$-action on $X\us$ is free, so the cohomology ring of $X//G$
can be identified with the equivariant cohomology ring
$H_G(X\uss,\bQ)$ and the Chow ring of $X//G$ with the equivariant Chow ring
$A_G(X\uss)_\bQ$.

\begin{remark}
  As Michel Brion explained to me, for any semisimple group $G$
  with a maximal torus $T$ and a $G$-variety $V$ that admits a
  $T$-invariant cell decomposition, the $G$-equivariant Chow ring
  $A^*_G(V)_\bQ$ and the $G$-equivariant cohomology ring
  $H^*_G(V,\bQ)$ coincide. Indeed, for the $T$-equivariant versions the
  cycle map $A^*_T(V)_\bQ \to H^*_T(V,\bQ)$ is an isomorphism by
  \cite{brion1997equivariant-chow}, and the $G$-equivariant versions
  $A^*_G$ and $H^*_G$ are the Weyl group invariant subrings of these.
\end{remark}

\begin{theorem}\label{thm:coh-ring}
  One has
  \begin{displaymath}
    A^*(X//G)_\bQ= H^*(X//G,\bQ) = 
    \bQ[z_1,\dotsc,z_7, c_2, c_3] / J
  \end{displaymath}
  with the generators of degree $(1,\dotsc,1,2,3)$ and the ideal $J$
  generated by the relations
  \begin{enumerate}
  \item $z_1^3 + c_2 z_1 + c_3$
  \item $\sigma_3(z_1,z_2,z_3,z_4,z_5) -
    \sigma_1(z_1,z_2,z_3,z_4,z_5)c_2 + c_3$
  \item $\sigma_4(z_1,z_2,z_3,z_4,z_5) -
    \sigma_2(z_1,z_2,z_3,z_4,z_5)c_2 + c_2^2 + \sigma_1(z_1,z_2,z_3,z_4,z_5)c_3$
  \item $(z_1^2z_2^2+z_2^2z_3^2+z_3^2z_1^2) +
    (z_1+z_2+z_3)(z_1z_2z_3-c_3) + (z_1^2 + z_2^2+ z_3^3)c_2 +
    c_2^2$
  \end{enumerate}
  and the ones obtained from them by permuting the variables
  $z_1,\dotsc, z_7$. Here, $\sigma_k$ are the elementary symmetric
  polynomials. 

  Moreover, the relations of types (1,2,4) and a
  single relation of type (3), or the sum of all relations of type (3),
  suffice. The relations of types (1,2) are independent.

  The dimensions of $A^{i}$ for $i=0,\dotsc,6$ are $(1,7,29,64,29,7,1)$.
\end{theorem}
\begin{proof}
  This is a direct application of Brion's paper
  \cite{brion1991cohomologie-equivariante}.  Let $T\subset G$ be the
  maximal torus. One has
  \begin{displaymath}
    A_T(\cdot) = S=\bQ[\eps_0,\eps_1,\eps_2]/(\eps_0+\eps_1+\eps_2), \quad
    A_G(\cdot) = S^W = \bQ[c_2,c_3], \quad W=S_3,
  \end{displaymath}
  where $c_2$ and $c_3$ are the Chern characters of the representation
  $V$ with $\bP^2=\bP(V)$, the elementary symmetric polynomials in
  $\eps_i$. Then the equivariant cohomology ring
  $A^*_T (X)$ equals $S[z_1,\dotsc,z_7]$ modulo the basic
  relations
  $(z_i+\epsilon_0)(z_i+\epsilon_1)(z_i+\epsilon_2)$. Similarly,
  $A^*_G (X)$ equals $S^W[z_1,\dotsc,z_7]$ modulo the basic
  relations $z_i^3 + c_2z_i+c_3$. Here, $z_i=c_1(\cO_{\bP^2}(1)$ for
  the $i$-th $\bP^2$,  cf. the remark below.

  Let $X\uss_T$ denote the set of semistable point for the action of
  $T$. Then by \cite[Thm.~2.1]{brion1991cohomologie-equivariante},
  $A^*_T(X\uss_T) = A^*_T(X) / I$, and the ideal $I$ is described with
  the help of the maximal unstable sets.  Up to the permutation by
  $S_3\times S_7$ of indices, they are
  \begin{displaymath}
    \{0\}^3,\  \{0,1,2\}^4 \quad\text{and}\quad \{1,2\}^5,\ \{0,1,2\}^2 
  \end{displaymath}
  which in a transparent way correspond to the instability conditions
  of $7$ lines on $\bP^2$ (see
  \cite{alexeev2009explicit-compactifications}): when three lines
  coincide or five lines pass through a common point.
  Then $I$ up to permutation by $S_3\times S_7$ is generated by the
  expressions 
  \begin{eqnarray*}
    &(z_1+\epsilon_1)(z_1+\epsilon_2)
    (z_2+\epsilon_1)(z_2+\epsilon_2)
    (z_3+\epsilon_1)(z_3+\epsilon_2)\\
    &\text{and} \quad
    (z_1+\epsilon_0)(z_2+\epsilon_0)(z_3+\epsilon_0)
      (z_4+\epsilon_0)(z_5+\epsilon_0)
  \end{eqnarray*}
  Then by \cite[Section~1.2]{brion1991cohomologie-equivariante} one
  has $A^*_G(X\uss) = A^*G(X) / p(I)$, where $p$ is the
  anti-symmetrization operator
  \begin{displaymath}
    p  = D\inv \sum_{w\in W} (-1)^{\sign(w)} w, \quad
    D =
  (\epsilon_0-\epsilon_1)(\epsilon_1-\epsilon_2)(\epsilon_2-\epsilon_0).
  \end{displaymath}
  Moreover, given generators $f_i$ of $I$ and an additive basis
  $\la g_j\ra$ for the harmonic module $\cH$, the ideal $p(I)$ is
  generated by $p(f_ig_j)$. For $G=\SL(3)$ the harmonic module is
  $\cH=\la 1, \epsilon_i, \epsilon_i^2-\epsilon_j^2, D\ra$.  The rest
  is a computation in sagemath \cite{sagemath}.
\end{proof}

\begin{remark}
  \cite{brion1991cohomologie-equivariante} follows the Grothendieck's
  convention for a projective space $\bP V$ as the space of
  $1$-dimensional quotient of $V$. We follow the convention that
  $\bP V$ is the space of lines in $V$, more common in the literature
  on equivariant cohomology. Then for us $z_i=c_1(\cO_{\bP^2}(1))$ and in
  \cite{brion1991cohomologie-equivariante} $z_i=c_1(\cO_{\bP^2}(-1))$.
\end{remark}

\subsection{$\GL(3,2)$-invariants}
\label{sec:gl32-invariants}

Denote by $s_k=\sigma_k(z_1,\dotsc,z_7)$ the elementary symmetric
polynomials in $z_1,\dotsc, z_7$. Also, denote by $s'_3$,
resp. $s''_3$, the sums of $z_iz_jz_k$ with distinct $i,j,k$ such that
the indices $i,j,k$ considered as points of the Fano plane
$\bP^2(\bF_2)$  are incident, resp. are not incident.
Obviously, $s'_3$ and $s''_3$ are $\GL(3,\bF_2)$-invariant but not
$S_7$-invariant. One has $s_3 = s'_3 + s''_3$ and denote $t=s''_3-4s'_3$.

\begin{theorem}\label{thm:inv-coh-ring}
  For the ring of invariants $A^*(X//G)^{\GL(3,2)}$,
  the dimensions of the invariants of degree $0,\dotsc,6$ are
  $(1,1,3,4,3,1,1)$ with an additive basis
  \begin{displaymath}
  1\quad s_1\quad s_1^2,c_2,s_2\quad s_1^3,c_2s_1,s_2s_1,t\quad
  c_2^2,c_2s_2,s_2^2\quad c_2^2s_1\quad c_2^3.     
  \end{displaymath}
  The algebra $A^*(X//G)^{\GL(3,2)}$ is generated by $s_1$,
  $c_2$, $s_2$, $t$ with relations
  \begin{eqnarray*}
    &&ts_1,\quad tc_2,\quad ts_2,\quad t^2 + 126 c_2^3,\quad \\
    &&45 s_1^4 - 1246 c_2^2 + 1090 c_2 s_2 - 240 s_2^2,\quad 
    15 c_2 s_1^2 - 28 c_2^2 + 35 c_2 s_2 - 10 s_2^2,\quad \\
    &&3 s_2 s_1^2 - 49 c_2^2 + 46 c_2 s_2 - 11 s_2^2,\quad 
    5 c_2 s_2 s_1 - 16 c_2^2 s_1,\quad 
    5 s_2^2 s_1 - 59 c_2^2 s_1.    
  \end{eqnarray*}
\end{theorem}
\begin{proof}
  The irreducible characters of $\GL(3,2)$ are $\chi_1$, $\chi_3$,
  $\chi_{\bar 3}$, $\chi_6$, $\chi_7$, $\chi_8$, where the subscript
  denotes the dimension.  $\chi_1$ is the trivial representation.
  There are two basic permutation representations of $S_7$ on the $7$
  variables $z_1,\dotsc z_7$ and on the $21$ monomials $z_iz_j$ with
  $i\ne j$. The induced $\GL(3,2)$-representations are
\begin{displaymath}
  \chi\uperm_7 = \chi_1 + \chi_6, \quad \chi_{21}\uperm = \chi_1+2\chi_6+\chi_8.
\end{displaymath}
From the relations of Theorem~\ref{thm:coh-ring}, the representations
on $A^{i}$ for $i=1,2,3$ are:
\begin{enumerate}
\item $A^1=\chi_7\uperm = \chi_1 + \chi_6$. Invariant:
  $s_1$. 
\item 
  $A^2=\Sym^2(\chi_7\uperm)+\chi_1\cdot c_2 = 3\chi_1 + 3\chi_6 + \chi_8.$
  Invariants: $s_1^2$, $s_2$, $c_2$.
\item From the generators and relations of type (1) and (2) we get:
  \begin{eqnarray*}
    A^3 &=& \Sym^3(\chi) + \chi_7\uperm\cdot c_2 + \chi_1\cdot c_3
            - \chi_7\uperm - \chi_{14}\uperm   \\ 
    &=& (4\chi_1 + 7\chi_6+2\chi_7+3\chi_8) + \chi_1 -
        (\chi_1+2\chi_6+\chi_8)\\
    &=& 4\chi_1 + 5\chi_6 + 2\chi_7 + 2\chi_8.
  \end{eqnarray*}
\end{enumerate}
From this and Poincare duality, for $\dim (H^{2i})^{\GL(3,2)}$ we get
$1,1,3,4,3,1,1$. By Hard Lefshetz, $s_1^{2i}H^{6-2i} \simeq
H^{6+2i}$. This gives additive bases in the invariant subspaces of
$H^8=A^4$, $H^{10}=A^5$, $H^{12}=A^6$, and we check that our choices
(leading to smaller formulas) also give bases.  For $H^6$, we get a
subspace $s_1\cdot (H^4)^{\GL(3,2)} = \la s_1^3, c_2s_1,
s_2s_1\ra$. Since these vectors are $S_7$-invariant and $t$ is not,
adding $t$ completes them to a basis of $(H^6)^{\GL(3,2)}$.

I checked the algebra relations and the fact that they suffice in sagemath.
\end{proof}

\begin{remark}
  The smaller subring of invariants $A^*(X//G)^{S_7}$ has
  $(1,1,3,3,3,1,1)$ for the dimensions of the graded pieces. It is
  generated by $s_1$, $c_2$, $s_2$ with the same relations, dropping
  those involving $t$.
\end{remark}

For the next section, I note the relation
$c_3 = \frac17   (2 s_1^3 - 6 s_2 s_1 + 17 c_2 s_1)$.

\subsection{Kappa classes}
\label{sec:kappa-classes}

The moduli stack $\cM\cam$ is a global quotient $[\oM:\GL(3,2)]$,
where $\oM = (\bP^2)^7//\SL(3,\bC)$ and let $f\colon \cY\to \oM$ be
the universal family of stable pairs $(Y, \sum_{i=1}^7 \frac12 B_i)$,
see \cite{alexeev2009explicit-compactifications}. The kappa classes on
$\oM$ are
\begin{displaymath}
  \kappa_l = f_* L^{l+2}, \quad\text{where}\quad
  L = \omega_{\cY/\oM} \left( \sum_{i=1}^7 \frac12 B_i \right).
\end{displaymath}

I am grateful to William Graham for explaining to me how to do
pushforward in equivariant cohomology, which is used in the proof of
the next theorem.

\begin{theorem}
  In the Chow ring $A^*(\oM, \bQ)$ described above, one has
  \begin{eqnarray*}
    2^2\kappa_0 &=& 1\\
    2^3\kappa_1 &=& 3s_1\\
    2^4\kappa_2 &=& 6s_1^2 - c_2\\
    2^5\kappa_3 &=& 10s_1^3 - 5s_1c_2 + c_3\\
    2^6\kappa_4 &=& 15s_1^4 - 15s_1^2c_2 + c_2^2 + 6s_1c_3\\
    2^7\kappa_5 &=& 21s_1^5 - 35s_1^3c_2 + 7s_1c_2^2 + 21s_1^2c_3 - 2c_2c_3\\
    2^8\kappa_6 &=& 28s_1^6 - 70s_1^4c_2 + 28s_1^2c_2^2 - c_2^3
    + 56s_1^3c_3 - 16s_1c_2c_3 + c_3^2
  \end{eqnarray*}
\end{theorem}
\begin{proof}
  Over $X = \bP (V)^7$ we have the universal family $X\times \bP (V^*)$
  with seven divisors $\cB_i$, and the family over $\oM$ is a quotient
  by a free action of $G=\SL(V)$. Therefore, it suffices to compute in
  $A^*_G(X)$. Denote $h=\cO_{\bP(V^*)}(1)$. Each $\cB_i$ is the incidence
  divisor in $\bP(V)\times\bP(V^*)$ and is linearly equivalent to
  $z_i+h$.  Therefore, 
  \begin{displaymath}
    L = \omega_{X\times\bP(V^*)/ X} \left( \sum_{i=1}^7 \frac12 B_i \right) = 
    -3h + \sum_{i=1}^7\frac12(h+z_i) = \frac12(h+s_1).
  \end{displaymath}
  By projection formula, to compute $f_*L^{i+2}$ it suffices to know
  $f_* h^k$ under the homomorphism $A^*_G(\bP(V^*))\to A^*_G(\cdot)=S$
  induced by the morphism
  $\bP(V^*)\to{\rm pt}$. For $s\in S^W$ one has $f_*(s)=f_*(hs)=0$,
  $f_*(h^2s)=s$, and the pushforwards of higher powers of $h$ follow by
  recursively using the basic relation $h^3 + c_2h - c_3=0$. The rest
  is an easy computation.
\end{proof}

Note that $s_1$ is the ample line bundle that comes with the GIT
quotient construction, the $\cO(1)$ on the $\Proj$ of the graded
algebra of invariants.

\bibliographystyle{amsalpha}

\newcommand{\etalchar}[1]{$^{#1}$}
\def\cprime{$'$}
\providecommand{\bysame}{\leavevmode\hbox to3em{\hrulefill}\thinspace}
\providecommand{\MR}{\relax\ifhmode\unskip\space\fi MR }
\providecommand{\MRhref}[2]{%
  \href{http://www.ams.org/mathscinet-getitem?mr=#1}{#2}
}
\providecommand{\href}[2]{#2}

\end{document}